\documentclass[a4paper,11pt]{article}
\usepackage{amsfonts}
\usepackage{mathrsfs}
\usepackage{indentfirst,latexsym,bm,amsmath,amssymb,amsthm}
\usepackage[dvips]{graphicx}

\makeatletter\@addtoreset{equation}{section} \makeatother
\newtheorem{theorem}{Theorem}[section]
\newtheorem{Lemma}{Lemma}[section]

\newtheorem{remark}{Remark}[section]

 

\title{ Lifespan of Classical Solutions to Quasilinear Wave Equations Outside of a Star-Shaped Obstacle in Four Space Dimensions }
\date{May~5, 2014}

\author{ Dongbing Zha\thanks{Corresponding author. School of Mathematical Sciences, Fudan University, Shanghai 200433, PR China.
{ E-mail address: ZhaDongbing@fudan.edu.cn(D.Zha), yizhou@fudan.edu.cn(Y.Zhou)}.}~,~~Yi Zhou}

\begin{document}

\maketitle
\begin{abstract}
 We study the initial-boundary value problem of
quasilinear wave equations outside of a star-shaped obstacle in four
space dimensions, in which the nonlinear term under consideration may explicitly depend on the unknown function itself.
By some new $L^{\infty}_{t}L^{2}_{x}$ and  weighted $L^{2}_{t,x}$ estimates for the unknown function itself, together with energy estimates and KSS estimates, for the quasilinear obstacle problem we obtain a lower bound of the lifespan $T_{\varepsilon}\geq \exp{(\frac{c}{\varepsilon^2})}$, which coincides with the sharp lower bound of lifespan estimate for the corresponding Cauchy problem.

\end{abstract}
\quad~\qquad{ {\bf  2010 MR Subject Classification }  35L05; 35L10; 35L20; 35L70.}

\pagestyle{plain} \pagenumbering{arabic}

\section{\bf Introduction and Main Result }\quad

This paper is devoted to study the lifespan of classical
solutions to the following initial-boundary value problem for nonlinear wave equations:
\begin{equation}\label{Quasilinear}
\left \{
\begin{array}{llll}
\Box u(t,x)=F(u,\partial u,\partial\nabla u), ~ (t,x)\in  \mathbb{R}^{+}\times \mathbb{R}^4\backslash  \mathcal {K} , \\
u|_{\partial \mathcal {K}}=0,\\
t=0:u=\varepsilon f, u_t=\varepsilon g,~x\in \mathbb{R}^4\backslash
\mathcal {K},\\

\end{array} \right.
\end{equation}
where
$\Box =\partial_{t}^2-\Delta$ is the wave operator, $\varepsilon>0$ is a small parameter, the obstacle $\mathcal
{K}\subset \mathbb{R}^4$ is compact, smooth and strictly star-shaped with respect to the origin, and $f,g$ in \eqref{Quasilinear} belongs to $C_c^{\infty}(\mathbb{R}^4\backslash \mathcal {K}).$
Moreover $(t,x)=(x_0,x_1,x_{2},x_{3},x_{4}), \partial_{\alpha}=\frac{\partial}{\partial x_{\alpha}}(\alpha=0,\cdots,4), \nabla=(\partial_{1},\partial_{2},\partial_{3},\partial_{4}),\partial=(\partial_{0},\nabla)$.
Let
\begin{equation}
\widehat{\lambda}=(\lambda; (\lambda_i),i=0,\cdots,4;
(\lambda_{ij}),i,j=0,\cdots,4, i+j\geq 1).
\end{equation}
Suppose that in a neighborhood of $\widehat{\lambda}=0$, say, for $|\widehat{\lambda}|\leq 1$, the nonlinear
term  $F=F(\widehat{\lambda})$ is a smooth
function satisfying
\begin{equation}
F(\widehat{\lambda})=\mathcal {O}(|\widehat{\lambda}|^2)
\end{equation}
and being affine with respect to $\lambda_{ij}(i,j=0,\cdots,4, i+j\geq 1)$ .\par
 Our aim is to study the lifespan of classical solutions to \eqref{Quasilinear}. By definition,
  the lifespan  $T_{\varepsilon}$ is the supremum of all $T>0$ such that
there exists a classical solution to \eqref{Quasilinear}  on $0\leq t\leq T,$ i.e.
\begin{align}
T_\varepsilon \stackrel{\mathrm{def}}{=} \sup \{T>0: \eqref{Quasilinear} {\text{ has a unique classical solution on}~[0,T] }\}.
\end{align}
\par
First, it is needed to illustrate why we consider the case of spatial dimension $n=4$. For this purpose, we have to review the history on the corresponding Cauchy problem in four space dimensions. In \cite{Hormander},  H\"{o}rmander considered the following Cauchy problem of nonlinear wave equations in four space dimensions:
\begin{align}\label{Cauchy}
 \begin{cases}
 \Box u(t,x)=F(u,\partial u,\partial\nabla u),~t\geq 0,~x\in \mathbb{R}^4,\\
 t=0:~u=\varepsilon f,~\partial_{t}u=\varepsilon g.
 \end{cases}
 \end{align}
Here in a neighborhood of $\widehat{\lambda}=0$, the nonlinear term $F$ is a smooth function with quadratic order with respect to its arguments. $f,g\in C_c^{\infty}(\mathbb{R}^4)$, and $\varepsilon>0$ is a small parameter. He proved that if $\partial_u^2F(0,0,0)=0$, then
\eqref{Cauchy} admits a unique global classical solution. For general $F$, he got a lower bound of the lifespan $T_{\varepsilon}\geq
  \exp({\frac{c}{\varepsilon}}), $ where $c$ is a positive constant independent of $\varepsilon$. But this result is not sharp. In \cite{Li 1}, Li and Zhou showed that H\"{o}rmander's estimate can be improved by
  \begin{align}
  T_{\varepsilon}\geq \exp{(\frac{c}{\varepsilon^2})}.
  \end{align}
  Li and Zhou's proof was simplified by Lindblad and Sogge in
\cite{Lindblad} later. Recently, the sharpness of Li and Zhou's estimate was shown by Takamura and  Wakasa in \cite{Takamura}
(see also Zhou and Han \cite{Zhou}). They proved that for the following Cauchy problem of semilinear wave equations:
\begin{align}\label{Cauchy1}
 \begin{cases}
 \Box u(t,x)=u^2,~t\geq 0,~x\in \mathbb{R}^4,\\
 t=0:~u=\varepsilon f,~\partial_{t}u=\varepsilon g,
 \end{cases}
 \end{align}
  the lifespan of classical solutions admits a upper bound: $T_{\varepsilon}\leq \exp{(\frac{c}{\varepsilon^2})}$ for some special functions $f,g\in C_c^{\infty}(\mathbb{R}^4)$. In fact, when the spatial dimension $n=4$, the equation in \eqref{Cauchy1} is just corresponding to the critical case of the Strauss conjecture, so it is the most difficult case to be handled. For the Strauss conjecture, we refer the reader to Strauss \cite{Strauss1} and the survey article Wang and Yu \cite{Wang}.
\par
The pioneering works by F. John and S. Klainerman open the field of lifespan estimate of classical solutions to the Cauchy problem of nonlinear wave equations. In other spatial dimensions, classical references can be found in John \cite{John0, John1, John}, John and~Klainerman \cite{Fr2}, Klainerman \cite{Klainerman1, Klainerman2, Klainerman, Klainerman4, Klainerman5,  Klainerman6}, Christodoulou \cite{Chris1}, H\"{o}rmander \cite{Hormander2, Hormander1}, Lindblad~\cite{Lindblad0}, Li and Chen~\cite{Li 2}, Li and Zhou~\cite{LiZhou2, LiZhou3, LiZhou4}, Alinhac \cite{Alinhac2, Alinhac3, Alinhac1} etc. . Especially, Klainerman's commutative vector field method in \cite{Klainerman} offer the basic framework for treating this kind of problem.

\par
For wave equations, a natural extension of the Cauchy problem is the initial-boundary value problem in exterior domains, which describes the wave propagation outside a bounded obstacle. Similarly to the Cauchy problem, the wave in exterior domains will propagate to the infinity(if the shape of the obstacle is sufficiently regular, say, for a star-shaped obstacle), But one should note the effect of the boundary condition, which will enhance the difficulty of the corresponding problem. In analogy with the Cauchy problem, we also want to get the lifespan estimate for classical solutions to the initial-boundary value problem in exterior domains.
\par
For the problem \eqref{Quasilinear} with obstacle
in four space dimensions , Du et al. \cite{Du1} established
a lower bound of the lifespan $T_{\varepsilon}\geq
  \exp({\frac{c}{\varepsilon}}). $  In \cite{Metcalfe2}, under the assumption
$\partial_{u}^2F(0,0,0)=0$, Metcalfe and Sogge  proved that \eqref{Quasilinear} has a unique global classical solution. Their results extend H\"{o}rmander's results from the Cauchy problem to the obstacle problem.
\par
In this paper, for the obstacle problem \eqref{Quasilinear} with general nonlinear term $F$, we get the lower bound of the lifespan $T_{\varepsilon}\geq \exp{(\frac{c}{\varepsilon^2})}$, which extends the result of Li and Zhou \cite{Li 1}  from the Cauchy problem to the obstacle problem. But the sharpness of this estimate is yet to be proved.
\par

In other dimensions, for the obstacle problem, when $n=3$, Du and Zhou \cite{Du2} showed that the analog of \eqref{Quasilinear} admits a unique classical solution with lifespan $T_{\varepsilon}\geq \frac{c}{\varepsilon^2}$(see also~\cite{He2}). For the special case where the nonlinear term
 does not explicitly depend on $u$, we refer the reader to ~\cite{KSS, KSS2, Metcalfe1} and the references therein. When $n\geq5$, Metcalfe and Sogge \cite{Metcalfe2} showed that the analog of \eqref{Quasilinear} has a unique global classical solution(see also Du et al. \cite{Du1}). The case of $n=2$ is still open.\par

 To prove our result, we will use Klainerman's commutative vector field method in \cite{Klainerman}. For the Cauchy problem, the Lorentz invariance of the wave operator is the key point of this method. However, for the obstacle problem, the Lorentz invariance does not hold. Another difficulty we encounter in the obstacle case is that, the homogeneous Dirichlet boundary condition is not preserved when some generalized derivatives act on the solution. To overcome these difficulties, Keel et al. \cite{KSS} established some weighted space-time estimates(KSS estiamtes) for first order derivatives of solutions to the linear wave equation. In \cite{Metcalfe1}, only by energy methods,  Metcalfe and Sogge established KSS estimates for perturbed linear wave equations on an exterior domain. By elliptic regularity estimates, high order KSS estimates(involve only general derivatives and spatial rotation operators) have been also established.  Using these estimates, they proved the long time existence for $n\geq 3$ when the nonlinear term does not depend explicitly on the unknown function. In this paper, we will prove our result by means of the framework of Metcalfe and Sogge \cite{Metcalfe1}. \par

  To handle the case that the nonlinear term depends explicitly on the unknown function, we will first prove a new $L^{\infty}_{t}L^{2}_{x}$ estimate for solutions to the Cauchy problem of linear wave equation in four space dimensions, based on the Morawetz estimate in \cite{Hidano2}.
  After that, starting from the $L^{\infty}_{t}L^{2}_{x}$ estimate and using the original method used to establish the KSS estimates in \cite{KSS}, combined with some pointwise estimates of the fundamental solution of wave operator in four space dimensions, we give a new weighted space-time $L^{2}_{t,x}$ estimate for the unknown function itself. The key point of the two estimates is that, on the right-hand side of these inequalities, we must take the $L^2$ norm with respect to the time variable. By the cut-off argument, we can extend these estimates to the case of obstacle problem of linear wave equation in four space dimensions. As to the estimate for the derivatives of solutions, we can use the energy estimate and KSS estimate in
  \cite{Metcalfe1}.
  \par

   Since we consider the problem with small initial data, the higher order terms have no essential influence on the
discussion of the lifespan of solutions with small amplitude,
without loss of generality, we assume that the nonlinear term $F$ can be
taken as
\begin{equation}\label{Nonlinearity}
F(u,\partial u,\partial\nabla u)= H(u,\partial u)+\sum^{4}_{\substack{\alpha,\beta=0\\ \alpha+\beta\geq 1}}\gamma^{\alpha \beta}(u,\partial u)\partial_{\alpha}\partial_{\beta}u,
\end{equation}
where~$H(u,\partial u)$  is a quadratic form, $\gamma^{\alpha\beta}$ is a linear form of $(u,\partial u)$ and satisfies the symmetry condition:
\begin{equation}\label{Coefficient}
\gamma^{\alpha\beta}(u,\partial u)=\gamma^{\beta\alpha}( u, \partial u), ~\alpha, \beta=0,1,\cdots,4, \alpha+\beta\geq 1.
\end{equation}\par
Without loss of generality, we may assume that the obstacle satisfying
\begin{align}\label{Obstacle}
\mathcal {K}\subset \mathbb{B}_{\frac{1}{2}}=\{x\in
\mathbb{R}^4:|x|< \frac{1}{2}\}.
\end{align}
\par
To solve \eqref{Quasilinear}, the data must be assumed to satisfy the relevant compatibility
conditions. Setting $J_k
u=\{\partial_{x}^a u: 0\leq |a|\leq k\},$
 we know that for a fixed $m$ and a
formal $H^m$  solution $u$ of \eqref{Quasilinear}, we can write $\partial_{t}^{k} u(0,\cdot)=\psi_k (J_k f,J_{k-1} g),0\leq
k\leq m, $
in which the compatibility functions $\psi_k(0\leq k\leq m)$ depend on the nonlinearity, $J_k f$ and $J_{k-1} g$. For
$(f,g)\in H^m \times H^{m-1}$, the compatibility condition simply requires that $\psi_k$ vanish
on $\partial \mathcal {K}$ for $0\leq k \leq m-1$. For smooth $(f , g)$, we say that the compatibility condition
is satisfied to infinite order if this vanishing condition holds for all $m$. For some further descriptions, see Keel et al. \cite{KSS22}.

The main theorem of this paper is the following
\begin{theorem}\label{Main12}
For the quasilinear initial-boundary problem \eqref{Quasilinear}, where the
obstacle $\mathcal {K}\subset \mathbb{R}^4$
 is compact, smooth and strictly star-shaped
with respect to the origin, and satisfies \eqref{Obstacle}, assume that initial data  $~f,g\in
C_c^{\infty}(\mathbb{R}^4\backslash \mathcal
{K})$, satisfies the compatibility conditions to infinite order,  and $F$ in \eqref{Quasilinear}
satisfies hypotheses (\ref{Nonlinearity}) and (\ref{Coefficient}).
Then for any given parameter $\varepsilon$ small enough,
\eqref{Quasilinear} admits a unique solution $u \in
C^{\infty}([0,T_\varepsilon)\times \mathbb{R}^4)$ with
\begin{equation}\label{sharpn}
T_{\varepsilon}\geq
  \exp({\frac{c}{\varepsilon^2}}),
\end{equation}
where $c$ is a positive constant independent of $\varepsilon$.
\end{theorem}
An outline of  this paper is as follows. In Section 2, we give some
notations. In Section 3, some $L^{\infty}_{t}L^{2}_{x}$ and weighted $L^2_{t,x}$ estimates for the wave equation in Minkowski space-time $\mathbb{R}^{1+4}$
  will be
established. In Section 4, we will give some estimates needed for obstacle problem. And then, the proof of Theorem \ref{Main12}
 will be presented in Section 5.

\section{ Some Notations}
Denote the spatial rotations
\begin{equation}
\Omega=(\Omega_{ij};1\leq i<j\leq 4),
\end{equation}
where
\begin{equation}
\Omega_{ij}=x_i\partial_j -x_j\partial_i , 1\leq i<j\leq 4,
\end{equation}
and the vector fields
\begin{equation}
Z=(\partial,\Omega)=(Z_1,Z_2,\cdots,Z_{11}).
\end{equation}
For any given multi-index~ $\varsigma=(\varsigma_1,\cdots,\varsigma_{11}),$ we denote
\begin{equation}
Z^{\varsigma}=Z_1^{\varsigma_1}\cdots Z_{11}^{\varsigma_{11}}.
\end{equation}

\par
As introduced in Li and Chen~\cite{Li 2}, we say that~$f\in
 L^{p,q}(\mathbb{R}^4),$ if
 \begin{align}
 f(r\omega)r^{\frac{3}{p}}\in L^p_r(0,\infty;
 L^q_w(S^3)),
 \end{align}
  where~$r=|x|, \omega=(\omega_1,\cdots,\omega_4)\in
 S^3,$ $S^3$ being the unit sphere in $\mathbb{R}^4$. For $1\leq p,q\leq +\infty,
 $ equipped with the norm
\begin{equation}
||f||_{L^{p,q}(\mathbb{R}^4)}\stackrel{\mathrm{def}}{=}||f(r\omega)r^{\frac{3}{p}}||_{L^p_r(0,\infty;
 L^q_w(S^3))},
\end{equation}
$L^{p,q}(\mathbb{R}^4)$ is a Banach space. It is easy to see that, if $p=q$, then $L^{p,q}(\mathbb{R}^4)$
 becomes the usual Lebesgue space $L^{p}(\mathbb{R}^4)$.

 \section{$L^{\infty}_{t}L^{2}_{x}$ and Weighted $L^2_{t,x}$ Estimates in Minkowski Space-time $\mathbb{R}^{1+4}$}
 \subsection{$L^{\infty}_{t}L^{2}_{x}$ Estimate}
 We first give a weighted Sobolev inequality which will be used in the proof of the
 ~$L^{\infty}_{t}L_x^2$ estimate of solutions.
\begin{Lemma}
If~$\frac{1}{2}< s_0< 2,$  then we have the estimate
\begin{equation}\label{Weighted Sobolev Estimates 0}
|||x|^{2-s_0} f||_{L^{\infty,2}(\mathbb{R}^4)}\leq C
||f||_{{\dot{H}}^{s_0}(\mathbb{R}^4)}
\end{equation}
and the corresponding dual estimate
\begin{equation}\label{Weighted Sobolev Estimates 1}
|| f||_{{\dot{H}}^{-s_0}(\mathbb{R}^4)}\leq C
|||x|^{s_0-2}f||_{L^{1,2}(\mathbb{R}^4)},
\end{equation}
where $C$ is a positive constant independent of $f$.
\end{Lemma}
 \begin{proof}
 By an 1-D Sobolev embedding, the localization technique and the dyadic decomposition, we can get \eqref{Weighted Sobolev Estimates 0}. For the details, see Li and Zhou \cite{Li 1} Theorem 2.10 or the Appendix of Wang and Yu~\cite{Wang}.
 \end{proof}
\begin{Lemma}\label{Morawetz}
(Morawetz Estimate) Let $v$ be the solution to the following problem:
\begin{equation}
\left \{
\begin{array}{ll}
\Box v (t,x)=0,(t,x)\in \mathbb{R}^{+}\times \mathbb{R}^4 ,\\
t=0:v=0,\partial_t v=g(x),x\in \mathbb{R}^4. \\
\end{array} \right.
\end{equation}
 For any given~$T>0$, denoting~$S_{T}=[0,T]\times \mathbb{R}^4$, we have the following weighted space-time estimate:
\begin{equation}
|||x|^{-s}v||_{L^2_{t,x}(S_{T})}\leq C||g||_{\dot{H}^{-(\frac{3}{2}-s)}(\mathbb{R}^4)},
\end{equation}
where $\frac{1}{2}< s< 1$ and $C$ is a positive constant independent of $T$.
\end{Lemma}
\begin{proof}~ See Hidano et al. \cite{Hidano2}
Lemma 3.1.
\end{proof}
\par

\begin{Lemma}\label{Dual estimate}
(Dual Estimate) Let $v$ be the solution to the following problem:
\begin{equation}
\left \{
\begin{array}{ll}
\Box v (t,x)=G(t,x),(t,x)\in \mathbb{R}^{+}\times \mathbb{R}^4 ,\\
t=0:v=0,\partial_t v=0,x\in \mathbb{R}^4. \\
\end{array} \right.
\end{equation}
Then for any given $T>0,$ we have
\begin{equation}\label{yu}
\sup_{0\leq t\leq
T}||v(t)||_{\dot{H}^{(\frac{3}{2}-s)}(\mathbb{R}^4)}\leq C|||x|^s
G||_{L^2(S_{T})},
\end{equation}
where $\frac{1}{2}< s< 1$ and $C$ is a positive constant independent of $T$.
\end{Lemma}
\begin{proof}
By Duhamel principle and Lemma \ref{Morawetz}, we have
\begin{equation}\label{Duhamel}
|||x|^{-s}v||_{L^2(S_T)}\leq C\int_0
^{T}||G(t)||_{\dot{H}^{-(\frac{3}{2}-s)}(\mathbb{R}^4)}dt.
\end{equation}
By duality,
\begin{equation}
 \sup_{0\leq t\leq
T}||v(t)||_{\dot{H}^{\frac{3}{2}-s}(\mathbb{R}^4)}
=\sup\{\int_{0}^{T}\int_{\mathbb{R}^4}v(t,x)P(t,x)dxdt:
\int_{0}^{T}||P(t)||_{\dot{H}^{-(\frac{3}{2}-s)}}dt=1\}.
\end{equation}
Let~$w$ satisfy
\begin{equation}
\left \{
\begin{array}{ll}
\Box w (t,x)=P(t,x),(t,x)\in \mathbb{R}^{+}\times \mathbb{R}^4 ,\\
t=T:w=0,\partial_t w=0,x\in \mathbb{R}^4. \\
\end{array} \right.
\end{equation}
 Integrating by parts, we have
\begin{align}
&\int_0^{T} \int_{\mathbb{R}^4}v(t,x)P(t,x)dxdt\nonumber\\
&=\int_0^{T} \int_{\mathbb{R}^4}v(t,x)\Box w(t,x)dxdt\nonumber\\
&=\int_0^{T} \int_{\mathbb{R}^4}\Box v(t,x) w(t,x)dxdt\nonumber\\
&=\int_0^{T} \int_{\mathbb{R}^4}G(t,x) w(t,x)dxdt\nonumber\\
&\leq C|||x|^s G||_{L^2(S_{T})}|||x|^{-s}
w||_{L^2(S_{T})}.
\end{align}
By~\eqref{Duhamel}, we get
\begin{equation}
|||x|^{-s}w||_{L^2(S_T)}\leq C\int_0
^{T}||P(t)||_{\dot{H}^{-(\frac{3}{2}-s)}(\mathbb{R}^4)}dt\leq C.
\end{equation}
So we finally obtain \eqref{yu}.
\end{proof}

\begin{Lemma}\label{L21}
($L^{\infty}_{t}L^2_{x}$ Estimate) Let $v$ satisfy
\begin{equation}\label{L234}
\left \{
\begin{array}{ll}
\Box v (t,x)=G(t,x),(t,x)\in \mathbb{R}^{+}\times \mathbb{R}^4 ,\\
t=0:~v=0,\partial_t v=0,x\in \mathbb{R}^4. \\
\end{array} \right.
\end{equation}
 Then for any given~$T>0,$ we have
\begin{align}\label{L22}
\sup_{0\leq t\leq T}|| v (t)||_{L^2(\mathbb{R}^4)}\leq
C|||x|^{-\frac{1}{2}}
 G||_{L^2([0,T];L^{1,2}(\mathbb{R}^4))},
\end{align}
where $C$ is a positive constant independent of $T$.
\end{Lemma}
\begin{proof} Let~$|D|=\sqrt{-\Delta}$, where $\Delta$ is the Laplacian operator on $\mathbb{R}^4$.  Acting ~$|D|^{-(\frac{3}{2}-s)}$ on both sides of \eqref{L234}, and noting Lemma~\ref{Dual estimate}, we get,
\begin{align}\label{Dual4}
\sup_{0\leq t\leq T}||v(t)||_{L^2(\mathbb{R}^4)}\leq C|||x|^s|D|^{-(\frac{3}{2}-s)}G||_{L^2(S_{T})}.
\end{align}
Since we have
\begin{align}
&|x|^s|D|^{-(\frac{3}{2}-s)}G(t,x)\nonumber\\
&=C|x|^s\int_{\mathbb{R}^4}\frac{G(t,y)}{~~|x-y|^{\frac{5}{2}+s}}dy\nonumber\\
&=C|x|^s\int_{|y|\leq \frac{|x|}{4}}\frac{G(t,y)}{~~|x-y|^{\frac{5}{2}+s}}dy+C|x|^s\int_{|y|\geq \frac{|x|}{4}}\frac{G(t,y)}{~~|x-y|^{\frac{5}{2}+s}}dy,
\end{align}
then we get the pointwise estimate
\begin{align}
&|x|^s||D|^{-(\frac{3}{2}-s)}G(t,x)|\nonumber\\
&\leq C\int_{\mathbb{R}^4}\frac{|G(t,y)|}{~~|x-y|^{\frac{5}{2}}}dy+C\int_{\mathbb{R}^4}\frac{|y|^s|G(t,y)|}{~~|x-y|^{\frac{5}{2}+s}}dy.
\end{align}
Consequently,
\begin{align}
&|||x|^s||D|^{-(\frac{3}{2}-s)}G(t)|||_{L^2_{x}(\mathbb{R}^4)}\nonumber\\
&\leq C|||G(t)|||_{\dot{H}^{-\frac{3}{2}}(\mathbb{R}^4)}+C|||x|^{s}|G(t)|||_{\dot{H}^{-(\frac{3}{2}-s)}(\mathbb{R}^4)}.
\end{align}
It follows from~\eqref{Weighted Sobolev Estimates 1} that
\begin{align}
&|||G(t)|||_{\dot{H}^{-\frac{3}{2}}(\mathbb{R}^4)}+|||x|^{s}|G(t)|||_{\dot{H}^{-(\frac{3}{2}-s)}(\mathbb{R}^4)}\nonumber\\
&\leq C|||x|^{-\frac{1}{2}}G(t)||_{L^{1,2}(\mathbb{R}^4)},
\end{align}
so we get
\begin{align}
|||x|^s|D|^{-(\frac{3}{2}-s)}G(t)||_{L_{x}^2(\mathbb{R}^4)}\leq C|||x|^{-\frac{1}{2}}G(t)||_{L^{1,2}(\mathbb{R}^4)}.
\end{align}
Consequently,
\begin{align}\label{RHS22}
|||x|^s|D|^{-(\frac{3}{2}-s)}G||_{L^2(S_{T})}\leq C|||x|^{-\frac{1}{2}}
 G||_{L^2([0,T];L^{1,2}(\mathbb{R}^4))}.
\end{align}
By \eqref{Dual4} and \eqref{RHS22}, we get the conclusion.
\end{proof}
\subsection{Weighted $L^2_{t,x}$ Estimate}
Now we will prove the weighted space-time $L^2_{t,x}$ estimate based on the $L^{\infty}_{t}L^2_{x}$ estimate given in Lemma \ref{L21}.  Our proof is inspired by the original proof of KSS inequality in \cite{KSS}. But in four space dimensions, there is no strong Huygens principle, so we must do some extra pointwise estimates by using the fundamental solution of wave equation. We will follow the argument used in Section 6.6 of Alinhac \cite{Alinhac}.
\begin{Lemma}\label{L22}
 (Weighted $L^2_{t,x}$ Estimate) Let~$v$ satisfy
\begin{equation}
\left \{
\begin{array}{ll}
\Box v (t,x)=G(t,x),(t,x)\in \mathbb{R}^{+}\times \mathbb{R}^4 ,\\
t=0:v=0,\partial_t v=0,x\in \mathbb{R}^4. \\
\end{array} \right.
\end{equation}
Then, for any given~$T>0$, we have
\begin{align}\label{Kss3}
&(\log(2+T))^{-1/2}||<x>^{-1/2}v||_{L^2_{t,x}(S_T)}+||<x>^{-3/4}v||_{L^2_{t,x}(S_T)}\nonumber\\
&\leq C|||x|^{-\frac{1}{2}}
 G||_{L^2([0,T];L^{1,2}(\mathbb{R}^4))},
\end{align}
where $C$ is a positive constant independent of $T$.
\end{Lemma}
\begin{proof}
Step 1. (Localization)~First we prove that
\begin{align}\label{Kss41}
||v||_{L^2([0,T];L^2(|x|\leq 2))}\leq C|||x|^{-\frac{1}{2}}
 G||_{L^2([0,T];L^{1,2}(\mathbb{R}^4))}.
\end{align}
Denoting
\begin{align}
R_k=\{(x,t): k\leq |x|+t<k+1\}, ~k=0,1,2,\cdots,
\end{align}
let $\chi_k$ be the characteristic function of~$R_k$, and $G_k=G\chi_k$, we have~$G=\sum_{k}G_k$. Let~$v_k$ satisfy
\begin{equation}
\left \{
\begin{array}{ll}
\Box v_k(t,x)=G_k(t,x),(t,x)\in \mathbb{R}^{+}\times \mathbb{R}^4 ,\\
t=0:v_k=0,\partial_t v_k=0,x\in \mathbb{R}^4. \\
\end{array} \right.
\end{equation}
 Obviously, ~$v=\sum_{k}v_k$.  In order to obtain~\eqref{Kss41}, without loss of generality, we assume that~$T\geq 10$ and~$T$ is an integer.
 We have
\begin{align}\label{Kss6}
&||v||^2_{L^2([8,T];L^2(|x|\leq 2))}\nonumber\\
&=\int_{8}^{T}\int_{|x|\leq 2}|v|^2dxdt\nonumber\\
&=\sum_{l=8}^{T-1}\int_{l}^{l+1}\int_{|x|\leq 2}|v|^2dxdt\nonumber\\
&\leq C\sum_{l=8}^{T-1}\int_{l}^{l+1}\int_{|x|\leq 2}|\sum_{|k-l|\leq 7}v_k|^2dxdt+C\sum_{l=8}^{T-1}\int_{l}^{l+1}\int_{|x|\leq 2}|\sum_{k> l+ 7}v_k|^2dxdt\nonumber\\
&+C\sum_{l=8}^{T-1}\int_{l}^{l+1}\int_{|x|\leq 2}|\sum_{k< l-7}v_k|^2dxdt.
\end{align}
Now we estimate the three terms on the right-hand side of~\eqref{Kss6}, respectively. For the first term, it follows from Lemma~\ref{L21} that
\begin{align}\label{Kss7}
&\sum_{l=8}^{T-1}\int_{l}^{l+1}\int_{|x|\leq 2}|\sum_{|k-l|\leq 7}v_k|^2dxdt\nonumber\\
&\leq C\sum_{l=8}^{T-1}\int_{l}^{l+1}\sum_{|k-l|\leq 7}\int_{|x|\leq 2}|v_k|^2dxdt\nonumber\\
&\leq C\sum_{l=8}^{T-1}\sum_{|k-l|\leq 7}\sup_{0\leq t\leq T}\int_{|x|\leq 2}|v_k|^2dx\nonumber\\
&\leq C\sum_{k=1}^{T+6}|||x|^{-\frac{1}{2}}
 G_k||^2_{L^2([0,T];L^{1,2}(\mathbb{R}^4))}\nonumber\\
&\leq C|||x|^{-\frac{1}{2}}
 G||^2_{L^2([0,T];L^{1,2}(\mathbb{R}^4))}.
\end{align}
 To get the last inequality in \eqref{Kss7}, we have used the Minkowski inequality. For the second term on the right-hand side of \eqref{Kss6}, noting the support of $G_k$, by Huygens principle we know that when $|x|\leq 2, l\leq t\leq l+1, k>l+7$, $v_k(t,x)=0$. Consequently,
\begin{align}\label{ler45}
\sum_{l=8}^{T-1}\int_{l}^{l+1}\int_{|x|\leq 2}|\sum_{k>l+7}v_k|^2dxdt=0.
\end{align}
 Now we deal with the third term on the right-hand side of \eqref{Kss6}.  We have
\begin{align}\label{Kss8}
v_k(t,x)=C\int_{R_k}\chi_{+}^{-\frac{3}{2}}((t-\tau)^2-|x-y|^2)G_k(\tau,y)dyd\tau,
\end{align}
where~$\chi_{+}^{-\frac{3}{2}}$ is the fundamental solution of wave operator in four space dimensions(see Section 6.2 of \cite{Hormander1}). Noting the support of $G_k$, we see that when~$|x|\leq 2, l\leq t\leq l+1, k<l-7, (\tau,y)\in R_k$ , we have
\begin{align}
t-\tau-|x-y|\geq t-\tau-|x|-|y|=t-|x|-(\tau+|y|)\geq C(l-k).
\end{align}
So by the properties of $\chi_{+}^{-\frac{3}{2}}$, we get
\begin{align}
&\chi_{+}^{-\frac{3}{2}}((t-\tau)^2-|x-y|^2)\nonumber\\
&\leq C((t-\tau)^2-|x-y|^2)^{-3/2}\nonumber\\
&\leq C(t-\tau-|x-y|)^{-3/2}(t-\tau+|x-y|)^{-3/2}\nonumber\\
&\leq C(t-|x|-(\tau+|y|))^{-3/2}(t-|x|-(\tau+|y|)+2|y|)^{-1+\delta}(t-\tau)^{-(1/2+\delta)}\nonumber\\
&\leq C(t-|x|-(\tau+|y|))^{-3/2}(t-|x|-(\tau+|y|))^{-1/2+\delta}|y|^{-1/2}(t-\tau)^{-(1/2+\delta)}\nonumber\\
&\leq C(l-k)^{-2+\delta}|y|^{-1/2}(t-\tau)^{-(1/2+\delta)},
\end{align}
where $\delta$ is a fixed real number and $0<\delta<\frac{1}{8}$. Using H\"{o}lder inequality, and noting that $t-\tau\geq C(l-k)\geq 7C,
2(\frac{1}{2}+\delta)>1$, we get
\begin{align}\label{pointse}
&|v_k(t,x)|\nonumber\\
&\leq C(l-k)^{-2+\delta}\int_{R_k}(t-\tau)^{-(1/2+\delta)}|y|^{-1/2}G_k(\tau,y)dyd\tau\nonumber\\
&\leq C(l-k)^{-2+\delta}|||x|^{-\frac{1}{2}}
G_k||_{L^2([0,T];L^{1,2}(\mathbb{R}^4))}.
\end{align}
It follows from Cauchy-Schwarz inequality and \eqref{pointse} that
\begin{align}
&|\sum_{k<l-7}v_k|^2\nonumber\\
&\leq  \sum_{k<l-7} (l-k)^{-2(\frac{1}{2}+\delta)}\sum_{k<l-7} (l-k)^{2(\frac{1}{2}+\delta)}|v_k|^2\nonumber\\
&\leq C\sum_{k<l-7} (l-k)^{1+2\delta}|v_k|^2\nonumber\\
&\leq C\sum_{k<l-7}(l-k)^{-3+4\delta}|||x|^{-\frac{1}{2}}
G_k||^2_{L^2([0,T];L^{1,2}(\mathbb{R}^4))}.
\end{align}
We then have
\begin{align}\label{Kss9}
&\sum_{l=8}^{T-1}\int_{l}^{l+1}\int_{|x|\leq 2}|\sum_{k< l-7}v_k|^2dxdt\nonumber\\
&\leq C\sum_{l=8}^{T-1}\sum_{k<l-7}(l-k)^{-3+4\delta}|||x|^{-\frac{1}{2}}
G_k||^2_{L^2([0,T];L^{1,2}(\mathbb{R}^4))}\nonumber\\
&\leq C\sum_{k=0}^{T}|||x|^{-\frac{1}{2}}
G_k||^2_{L^2([0,T];L^{1,2}(\mathbb{R}^4))}\nonumber\\
&\leq C|||x|^{-\frac{1}{2}}
G||^2_{L^2([0,T];L^{1,2}(\mathbb{R}^4))},
\end{align}
here, Minkowski inequality is used in the last step of \eqref{Kss9}.
By \eqref{Kss6}--\eqref{ler45} and \eqref{Kss9}, we get \eqref{Kss41}.

Step 2. (Scaling) By scaling argument, we will pass from the inequality of step 1, i.e., \eqref{Kss41}, to the following inequality in an annulus:
\begin{align}\label{Kss100}
|||x|^{-\frac{1}{2}}v||_{L^2([0,T];L^2(R\leq |x|\leq 2R))}
\leq C|||x|^{-\frac{1}{2}}G||_{L^2([0,T];L^{1,2}(\mathbb{R}^4))},
\end{align}
where $R>0$, $C$ is a positive constant independent of $R$. Denoting
\begin{align}
v_{R}(x,t)=v(Rx,Rt),~G_{R}(t,x)=R^2G(Rx,Rt),
\end{align}
$v_{R}$ satisfies
\begin{equation}
\left \{
\begin{array}{ll}
\Box v_{R} (t,x)=G_{R}(t,x),(t,x)\in \mathbb{R}^{+}\times \mathbb{R}^4 ,\\
t=0:v_{R}=0,\partial_t v_{R}=0,x\in \mathbb{R}^4. \\
\end{array} \right.
\end{equation}
It follows from \eqref{Kss41} that
\begin{align}\label{Kss4}
||v_{R}||_{L^2([0,\frac{T}{R}];L^2(1\leq |x|\leq 2))}\leq C|||x|^{-\frac{1}{2}}
G_{R}||_{L^2([0,\frac{T}{R}];L^{1,2}(\mathbb{R}^4))}.
\end{align}
By simple calculation, we have
\begin{align}\label{Kss4}
||v_{R}||_{L^2([0,\frac{T}{R}];L^2(1\leq |x|\leq 2))}=R^{-5/2}||v||_{L^2([0,{T}];L^2(R\leq |x|\leq 2R))},
\end{align}
\begin{align}
|||x|^{-\frac{1}{2}}
G_{R}||_{L^2([0,\frac{T}{R}];L^{1,2}(\mathbb{R}^4))}=R^{-2}|||x|^{-\frac{1}{2}}
G||_{L^2([0,{T}];L^{1,2}(\mathbb{R}^4))},
\end{align}
so we get
\begin{align}\label{Kss10}
|||x|^{-1/2}v||_{L^2([0,T];L^2(R\leq |x|\leq 2R))}
\leq C|||x|^{-\frac{1}{2}}
G||_{L^2([0,{T}];L^{1,2}(\mathbb{R}^4))}.
\end{align}
\par
Step 3. (Dyadic decomposition) To get~\eqref{Kss3}
, we will prove
\begin{align}\label{Kss12}
(\log(2+T))^{-1/2}||<x>^{-1/2}v||_{L^2_{t,x}(S_T)}
\leq  C|||x|^{-\frac{1}{2}}G||_{L^2([0,T];L^{1,2}(\mathbb{R}^4))},
\end{align}
and
\begin{align}\label{Kss131}
||<x>^{-3/4}v||_{L^2_{t,x}(S_T)}
\leq  C|||x|^{-\frac{1}{2}}G||_{L^2([0,T];L^{1,2}(\mathbb{R}^4))},
\end{align}
respectively.
To get \eqref{Kss12}, noting that by Lemma \ref{L21} we have
\begin{align}
&||<x>^{-1/2}v||_{L^2([0,T]; L^2(|x|\geq T))}\\
&\leq C(\log(2+T))^{1/2}\sup_{0\leq t\leq T }||v(t)||_{L^2(\mathbb{R}^4)}\\
&\leq C(\log(2+T))^{1/2}|||x|^{-\frac{1}{2}}G||_{L^2([0,T];L^{1,2}(\mathbb{R}^4))},
\end{align}
 we need only to prove
\begin{align}\label{Kss14}
(\log(2+T))^{-1/2}||<x>^{-1/2}v||_{L^2([0,T]; L^2(|x|\leq T))}
\leq C|||x|^{-\frac{1}{2}}G||_{L^2([0,T];L^{1,2}(\mathbb{R}^4))}.
\end{align}
Taking~$N$ such that~$2^{N}<T\leq 2^{N+1}$, by \eqref{Kss41} and \eqref{Kss100}, we get
\begin{align}
&\int_{0}^{T}\int_{|x|\leq T}(1+|x|)^{-1}|v|^2dxdt\nonumber\\
&=\int_{0}^{T}\int_{|x|\leq 1}(1+|x|)^{-1}|v|^2dxdt+\sum_{k=0}^{N}\int_{0}^{T}\int_{2^{k}\leq |x|\leq 2^{k+1}}(1+|x|)^{-1}|v|^2dxdt\nonumber\\
&\leq C (N+2)|||x|^{-\frac{1}{2}}G||^2_{L^2([0,T];L^{1,2}(\mathbb{R}^4))}\nonumber\\
&\leq C\log(2+T)|||x|^{-\frac{1}{2}}G||^2_{L^2([0,T];L^{1,2}(\mathbb{R}^4))}.
\end{align}
This completes the proof of \eqref{Kss12}. Now we prove \eqref{Kss131}. Taking $R=2^{k}(k=0,1,2,\cdots)$ in \eqref{Kss100},  we have
\begin{align}\label{Kss18}
\int_{0}^{T}\int_{2^{k}\leq |x|\leq 2^{k+1}}(1+|x|)^{-1}|v|^2dxdt
\leq C|||x|^{-\frac{1}{2}}G||^2_{L^2([0,T];L^{1,2}(\mathbb{R}^4))}.
\end{align}
 Multiplying $2^{-\frac{k}{2}}$ on both sides of \eqref{Kss18}, and then summing up with respect to $k$, we get
\begin{align}\label{Kss19}
||<x>^{-3/4}v||_{L^2([0,T];L^2(|x|\geq 1))}
\leq  C||x|^{-\frac{1}{2}}G||_{L^2([0,T];L^{1,2}(\mathbb{R}^4))}.
\end{align}
Combining \eqref{Kss19} with \eqref{Kss41}, we can get \eqref{Kss131}.
\end{proof}
\subsection{Higher Order $L^{\infty}_{t}L^{2}_{x}$ and Weighted $L^2_{t,x}$ Estimates}
By Lemma \ref{L21} and Lemma \ref{L22}, we have the following
\begin{Lemma}\label{L23}
 Let $v$ satisfy
\begin{equation}
\left \{
\begin{array}{ll}
\Box v (t,x)=G(t,x),(t,x)\in \mathbb{R}^{+}\times \mathbb{R}^4 ,\\
t=0:v=0,\partial_t v=0,x\in \mathbb{R}^4. \\
\end{array} \right.
\end{equation}
Then for  any given $T>0$ we have
\begin{align}\label{Kss334}
&\sup_{0\leq t\leq T}||v(t)||_{L^2(\mathbb{R}^4)}+(\log(2+T))^{-1/2}||<x>^{-1/2}v||_{L^2_{t,x}(S_T)}\nonumber\\
&+||<x>^{-3/4}v||_{L^2_{t,x}(S_T)}\nonumber\\
&\leq C|||x|^{-\frac{1}{2}}
 G||_{L^2([0,T];L^{1,2}(\mathbb{R}^4))},
\end{align}
where $C$ is a positive constant independent of $T$.
\end{Lemma}
Since the wave operator $\square$ commutes with $Z$,  it is not hard to get the higher order versions of \eqref{Kss334} as follows.
\begin{Lemma}\label{L24}
 Let~$v$ satisfy
\begin{equation}
\left \{
\begin{array}{ll}
\Box v (t,x)=G(t,x),(t,x)\in \mathbb{R}^{+}\times \mathbb{R}^4 ,\\
t=0:v=0,\partial_t v=0,x\in \mathbb{R}^4. \\
\end{array} \right.
\end{equation}
Then for any given $T>0$ and $N=0,1,2,\cdots$, we have
\begin{align}\label{Kss1323}
&\sum_{|\mu|\leq N}\sup_{0\leq t\leq T}||Z^{\mu}v(t)||_{L^2(\mathbb{R}^4)}+\sum_{|\mu|\leq N}(\log(2+T))^{-1/2}||<x>^{-1/2}Z^{\mu}v||_{L^2_{t,x}(S_T)}\nonumber\\
&+\sum_{|\mu|\leq N}||<x>^{-3/4}Z^{\mu}v||_{L^2_{t,x}(S_T)}\nonumber\\
&\leq C\sum_{|\mu|\leq N}|||x|^{-\frac{1}{2}}
 Z^{\mu}G||_{L^2([0,T];L^{1,2}(\mathbb{R}^4))}
\end{align}
and
\begin{align}\label{Kss1329}
&\sum_{|\mu|\leq N}\sup_{0\leq t\leq T}||\partial_{t,x}^{\mu}v(t)||_{L^2(\mathbb{R}^4)}+\sum_{|\mu|\leq N}(\log(2+T))^{-1/2}||<x>^{-1/2}\partial_{t,x}^{\mu}v||_{L^2_{t,x}(S_T)}\nonumber\\
&+\sum_{|\mu|\leq N}||<x>^{-3/4}\partial_{t,x}^{\mu}v||_{L^2_{t,x}(S_T)}\nonumber\\
&\leq C\sum_{|\mu|\leq N}|||x|^{-\frac{1}{2}}
 \partial_{t,x}^{\mu}G||_{L^2([0,T];L^{1,2}(\mathbb{R}^4))},
\end{align}
where $C$ is a positive constant independent of $T$.
\end{Lemma}
\section{Some Estimates Outside of a Star-Shaped Obstacle}
\subsection{Higher Order $L^{\infty}_{t}L^{2}_{x}$ and Weighted $L^2_{t,x}$ Estimates}
 In this section, we will give some estimates outside of a star-shaped obstacle which is needed in the proof of lifespan estimate. By Lemma \ref{L24} and a cutoff argument, $L^{\infty}_tL^2_{x}$ and weighted $L^2_{t,x}$ estimates for the unknown function itself can be established.
\begin{Lemma}\label{KS5}
Let $v\in C^{\infty}(\mathbb{R}^{+}\times \mathbb{R}^4\backslash \mathcal {K})$ satisfy
\begin{align}\label{exterior}
\begin{cases}
\square v(t,x)=G(t,x),~(t,x)\in \mathbb{R}^{+}\times\mathbb{R}^4\backslash \mathcal {K},\\
v(t,x)=0,~x\in \partial  \mathcal {K},\\
t=0:~v=0,~\partial_{t}v=0,
\end{cases}
\end{align}
where the obstacle $\mathcal {K}$ is bounded, smooth and strictly star-shaped with
respect to the origin. Assume that \eqref{Obstacle} holds. For  any given $T>0$, denoting $S_T=[0,T]\times \mathbb{R}^4\backslash \mathcal {K}$ and $N=1,2\cdots,$ we have
\begin{align}\label{Kss14}
&\sum_{|\mu|\leq N}\sup_{0\leq t\leq T}||Z^{\mu}v(t)||_{L^2(\mathbb{R}^4\backslash \mathcal {K})}+\sum_{|\mu|\leq N}(\log(2+T))^{-1/2}||<x>^{-1/2}Z^{\mu}v||_{L^2_{t,x}(S_T)}\nonumber\\
&+\sum_{|\mu|\leq N}||<x>^{-3/4}Z^{\mu}v||_{L^2_{t,x}(S_T)}\nonumber\\
&\leq C\sum_{|\alpha|\leq N-1}\sup_{0\leq t\leq T}||\partial_{t,x}^{\mu}v'(t)||_{L^2(|x|\leq 1)}+C\sum_{|\mu|\leq N}||\partial_{t,x}^{\mu}v'||_{L^2_{t,x}([0,T]\times \{|x|\leq 1\})}\nonumber\\
&+C\sum_{|\mu|\leq N}||<x>^{-\frac{1}{2}}
 Z^{\mu}G||_{L^2([0,T];L^{1,2}(|x|>\frac{3}{4}))}
\end{align}
and
\begin{align}\label{Kss1123}
&\sum_{|\mu|\leq N}\sup_{0\leq t\leq T}||\partial_{t,x}^{\mu}v(t)||_{L^2(\mathbb{R}^4\backslash \mathcal {K})}+\sum_{|\mu|\leq N}(\log(2+T))^{-1/2}||<x>^{-1/2}\partial_{t,x}^{\mu}v||_{L^2_{t,x}(S_T)}\nonumber\\
&+\sum_{|\mu|\leq N}||<x>^{-3/4}\partial_{t,x}^{\mu}v||_{L^2_{t,x}(S_T)}\nonumber\\
&\leq C\sum_{|\mu|\leq N-1}\sup_{0\leq t\leq T}||\partial_{t,x}^{\mu}v'(t)||_{L^2(|x|\leq 1)}+C\sum_{|\mu|\leq N}||\partial_{t,x}^{\mu}v'||_{L^2_{t,x}([0,T]\times \{|x|\leq 1\})}\nonumber\\
&+C\sum_{|\mu|\leq N}||<x>^{-\frac{1}{2}}
 \partial_{t,x}^{\mu}G||_{L^2([0,T];L^{1,2}(|x|>\frac{3}{4}))},
\end{align}
where $C$ is a positive constant independent of $T$.
\end{Lemma}
\begin{proof}
 Noting that $v$ satisfies the homogeneous Dirichlet boundary condition, by Poincar\'{e} inequality we have
\begin{align}
\sup_{0\leq t\leq T}||v(t)||_{L^2(|x|\leq 1)}\leq C\sup_{0\leq t\leq T}||v'(t)||_{L^2(|x|\leq 1)}.
\end{align}
Then we get
\begin{align}\label{Poin1}
&\sum_{|\mu|\leq N}\sup_{0\leq t\leq T}||Z^{\mu}v(t)||_{L^2(|x|\leq 1)}\nonumber\\
&\leq C\sum_{|\mu|\leq N}\sup_{0\leq t\leq T}||\partial_{t,x}^{\mu}v(t)||_{L^2(|x|\leq 1)}\nonumber\\
&\leq C\sum_{|\mu|\leq N-1}\sup_{0\leq t\leq T}||\partial_{t,x}^{\mu}v'(t)||_{L^2(|x|\leq 1)}+C\sup_{0\leq t\leq T}||v(t)||_{L^2(|x|\leq 1)}\nonumber\\
&\leq C\sum_{|\mu|\leq N-1}\sup_{0\leq t\leq T}||\partial_{t,x}^{\mu}v'(t)||_{L^2(|x|\leq 1)}+C\sup_{0\leq t\leq T}||v'(t)||_{L^2(|x|\leq 1)}\nonumber\\
&\leq C\sum_{|\mu|\leq N-1}\sup_{0\leq t\leq T}||\partial_{t,x}^{\mu}v'(t)||_{L^2(|x|\leq 1)}.\nonumber\\
\end{align}
Similarly, we have
\begin{align}\label{Poin2}
\sum_{|\mu|\leq N}||Z^{\mu}v||_{L^2_{t,x}([0,T]\times \{|x|\leq 1 \})}\leq C\sum_{|\mu|\leq N-1}||\partial_{t,x}^{\mu}v'||_{L^2_{t,x}([0,T]\times \{|x|\leq 1 \})}.
\end{align}
So
\begin{align}\label{Kss198}
&\sum_{|\mu|\leq N}\sup_{0\leq t\leq T}||Z^{\mu}v(t)||_{L^2(|x|\leq 1)}+\sum_{|\mu|\leq N}(\log(2+T))^{-1/2}||<x>^{-1/2}Z^{\mu}v||_{L^2([0,T]\times\{|x|\leq 1\})}\nonumber\\
&+\sum_{|\mu|\leq N}||<x>^{-3/4}Z^{\mu}v||_{L^2([0,T]\times\{|x|\leq 1\})}\nonumber\\
&\leq C\sum_{|\mu|\leq N}\sup_{0\leq t\leq T}||Z^{\mu}v(t)||_{L^2(|x|\leq 1)}+C\sum_{|\mu|\leq N}||Z^{\mu}v||_{L^2_{t,x}([0,T]\times \{|x|\leq 1 \})}\nonumber\\
&\leq C\sum_{|\mu|\leq N-1}\sup_{0\leq t\leq T}||\partial_{t,x}^{\mu}v'(t)||_{L^2(|x|\leq 1)}+C\sum_{|\mu|\leq N-1}||\partial_{t,x}^{\mu}v'||_{L^2_{t,x}([0,T]\times \{|x|\leq 1 \})}.
\end{align}
Taking a smooth cutoff function $\rho$ such that
\begin{align}
\rho(x)=
\begin{cases}
1,~|x|\geq 1,\\
0,~|x|\leq \frac{3}{4},
\end{cases}
\end{align}
and denoting $\phi=\rho v$, $\phi$ satisfies the following wave equation in the whole space:
\begin{align}
\square \phi=\rho G-2\nabla \rho\cdot \nabla v-\Delta \rho v:=\widetilde{G}.
\end{align}
It follows from \eqref{Kss1323} and Poincar\'{e} inequality that
\begin{align}\label{Kss15}
&\sum_{|\mu|\leq N}\sup_{0\leq t\leq T}||Z^{\mu}v(t)||_{L^2(|x|\geq 1)}+\sum_{|\mu|\leq N}(\log(2+T))^{-1/2}||<x>^{-1/2}Z^{\mu}v||_{L^2_{t,x}([0,T]\times {|x|\geq 1})}\nonumber\\
&+\sum_{|\mu|\leq N}||<x>^{-3/4}Z^{\mu}v||_{L^2_{t,x}([0,T]\times {|x|\geq 1})}\nonumber\\
&\leq \sum_{|\mu|\leq N}\sup_{0\leq t\leq T}||Z^{\mu}\phi(t)||_{L^2(\mathbb{R}^4)}+\sum_{|\mu|\leq N}(\log(2+T))^{-1/2}||<x>^{-1/2}Z^{\mu}\phi||_{L^2_{t,x}(S_T)}\nonumber\\
&+\sum_{|\mu|\leq N}||<x>^{-3/4}Z^{\mu}\phi||_{L^2_{t,x}(S_T)}\nonumber\\
&\leq C\sum_{|\mu|\leq N}|||x|^{-\frac{1}{2}}
 Z^{\mu}\widetilde{G}||_{L^2([0,T];L^{1,2}(\mathbb{R}^4))}\nonumber\\
 &\leq C\sum_{|\mu|\leq N}||\partial_{t,x}^{\mu}v'||_{L^2_{t,x}([0,T]\times \{|x|\leq 1\})}+
C\sum_{|\mu|\leq N}||<x>^{-\frac{1}{2}}
 Z^{\mu}G||_{L^2([0,T];L^{1,2}(|x|>\frac{3}{4}))}.
 \end{align}
 By \eqref{Kss198} and \eqref{Kss15}, we get \eqref{Kss14}. Similarly, we can get \eqref{Kss1123}.
\end{proof}
We point out that two localized linear terms
\begin{align}\sum_{|\mu|\leq N-1}\sup_{0\leq t\leq T}||\partial_{t,x}^{\mu}v'(t)||_{L^2(|x|\leq 1)}~\text{and}~\sum_{|\mu|\leq N}||\partial_{t,x}^{\mu}v'||_{L^2_{t,x}([0,T]\times \{|x|\leq 1\})}
 \end{align}
 appear on the right-hand side of ~\eqref{Kss14} and \eqref{Kss1123}. These terms can be estimated by energy estimate and KSS inequalities, which have been obtained by Metalfe and Sogge in \cite{Metcalfe1}. Now we list these estimates without proofs in the next section(for details, see Lemma 3.2, Lemma 3.3, Lemma 5.2 and Lemma 5.3 in \cite{Metcalfe1}).
\subsection{Energy Estimate and KSS Estimate}
\begin{Lemma}\label{KSS123}
Let $w$ satisfy
\begin{equation}\label{Perturbation}
\left \{
\begin{array}{llll}
\Box_h w(t,x)=Q(t,x), ~ (x,t)\in \mathbb{R}^{+}\times \mathbb{R}^4\backslash  \mathcal {K}  , \\
w|_{\partial \mathcal {K}}=0,\\
\end{array} \right.
\end{equation}
where
\begin{align}
\Box_h w=(\partial_t^2 -\Delta)w
+\sum_{\alpha,\beta=0}^{4}h^{\alpha\beta}(t,x)\partial_{\alpha}\partial_{\beta}w.
\end{align}
 Assume that, without loss of generality, $h^{\alpha\beta}$ satisfy the symmetry conditions
\begin{equation}
h^{\alpha\beta}=h^{\beta\alpha},
\end{equation}
and the smallness condition
\begin{equation}
|h|\ll 1.
\end{equation}
Here we denote
\begin{equation}
|h|=\sum_{\alpha,\beta=0}^{4}|h^{\alpha\beta}|,
\end{equation}
\begin{equation}
|\partial
h|=\sum_{\alpha,\beta,\gamma=0}^{4}|\partial_{\gamma}h^{\alpha\beta}|.
\end{equation}
Then we have
\begin{align}\label{KSS21}
&\sup_{0\leq t\leq T}\sum_{\substack{|\mu|\leq N\\|a|=0,1}}||\partial^{\mu}\partial^{a}w'(t)||^2_{L^2(\mathbb{R}^4\backslash \mathcal {K})}+
\sum_{\substack{|\mu|\leq N\\|a|=0,1}}||<x>^{-\frac{3}{4}}\partial^{\mu}\partial^{a}w'||^2_{L^2(S_T)}\nonumber\\
&\leq C\sum_{\substack{|\mu|\leq N\\|a|=0,1}}||\partial^{\mu}\partial^{a}w'(0)||^2_{L^2(\mathbb{R}^4\backslash \mathcal {K})}
+C\sum_{\substack{|{\mu}|,|\nu|\leq N\\|a|,|b|=0,1}}\int_0^{T}\int_{\mathbb{R}^4\backslash \mathcal {K}}(|\partial^{\mu}\partial^{a}w'|+\frac{|\partial^{\mu}\partial^{a}w|}{r})|\partial^{\nu}\partial^{b}Q|dxdt\nonumber\\
&+C\sum_{\substack{|{\mu}|,|\nu|\leq N\\|a|,|b|=0,1}}\int_0^{T}\int_{\mathbb{R}^4\backslash \mathcal {K}}(|\partial h|+\frac{|h|}{r})|\partial^{\mu}\partial^{a}w'|(|\partial^{\nu}\partial^{b}w'|+\frac{|\partial^{\nu}\partial^{b}w|}{r})dxdt\nonumber\\
&+C\sum_{\substack{|{\mu}|,|\nu|\\|a|,|b|=0,1}}\sum_{\alpha,\beta=0}^{4} \int_0^{T}\int_{\mathbb{R}^4\backslash \mathcal {K}}(|\partial^{\mu}\partial^{a}w'|+\frac{|\partial^{\mu}\partial^{a}w|}{r})|[h^{\alpha\beta}\partial_{\alpha\beta}, \partial^{\nu}\partial^{b}]w|dxdt\nonumber\\
&+C\sum_{\substack{|\mu|\leq N-1\\|a|=0,1}}||\partial^{\mu}\partial^{a}\square w||^2_{L^2_{t,x}(S_{T})}+C\sum_{\substack{|\mu|\leq N-1\\|a|=0,1}}||\partial^{\mu}\partial^{a}\square w||^2_{L^{\infty}([0,T];L^{2}(\mathbb{R}^4\backslash \mathcal {K}))}
\end{align}
and
\begin{align}\label{KSS22}
&\sup_{0\leq t\leq T}\sum_{\substack{|\mu|\leq N\\|a|=0,1}}||Z^{\mu}\partial^{a}w'(t)||^2_{L^2(\mathbb{R}^4\backslash \mathcal {K})}+ \sum_{\substack{|\mu|\leq N\\|a|=0,1}}||<x>^{-\frac{3}{4}}Z^{\mu}\partial^{a}w'||^2_{L^2(S_T)}\nonumber\\
&\leq C\sum_{\substack{|\mu|\leq N\\|a|=0,1}}||Z^{\mu}\partial^{a}w'(0)||^2_{L^2(\mathbb{R}^4\backslash \mathcal {K})}
+C\sum_{\substack{|\mu|,|\nu|\leq N\\|a|,|b|=0,1}}\int_0^{T}\int_{\mathbb{R}^4\backslash \mathcal {K}}(|Z^{\mu}\partial^{a}w'|+\frac{|Z^{\mu}\partial^{a}w|}{r})|Z^{\nu}\partial^{b}Q|dxdt\nonumber\\
&+C\sum_{\substack{|\mu|,|\nu|\leq N\\|a|,|b|=0,1}}\int_0^{T}\int_{\mathbb{R}^4\backslash \mathcal {K}}(|\partial h|+\frac{|h|}{r})|Z^{\mu}\partial^{a}w'|(|Z^{\nu}\partial^{b}w'|+\frac{|Z^{\nu}\partial^{b}w|}{r})dxdt\nonumber\\
&+C\sum_{\substack{|\mu|,|\nu|\leq N\\|a|,|b|=0,1}}\sum_{\alpha,\beta=0}^{4} \int_0^{T}\int_{\mathbb{R}^4\backslash \mathcal {K}}(|Z^{\mu}\partial^{a}w'|+\frac{|Z^{\mu}\partial^{a}w|}{r})|[h^{\alpha\beta}\partial_{\alpha\beta}, Z^{\nu}\partial^{b}]w|dxdt\nonumber\\
&+C\sum_{\substack{|\mu|\leq N+1\\|a|=0,1}}||\partial^{\mu}_{x}\partial^{a}w'||^2_{L_{t,x}^{2}([0,T]\times \{|x|\leq 1\})}+
C\sum_{\substack{|\mu|\leq N+1\\|a|=0,1}}||\partial^{\mu}_{x}\partial^{a}w'||^2_{L^{\infty}([0,T]; L^{2}(\{|x|\leq 1\}))},
\end{align}
where $[~,~]$ stands for the Poisson's bracket, and $C$ is a positive constant independent of $T$.
\end{Lemma}
\begin{remark}
The energy estimate and KSS estimates \cite{Metcalfe1} involve only higher order estimates of the first order derivatives of the unknown function.
Higher order estimates of the second order derivatives of the unknown function can be obtained by the same method.
\end{remark}
\subsection{Decay Estimate}
\begin{Lemma}\label{spatial decay}
Assume that $f\in C^{\infty}(\mathbb{R}^4\backslash \mathcal {K})$ and $f$ vanishes for large $x$. Then for all $x\in \mathbb{R}^4\backslash \mathcal {K}$ we have
\begin{align}\label{01}
<r>^{\frac{3}{2}}|f(x)|\leq C
\sum_{|a|\leq 3}||Z^{a}f||_{L^2(\mathbb{R}^4\backslash\mathcal {K})},
\end{align}
where $r=|x|, <r>=(1+r^2)^{\frac{1}{2}}$, and $C$ is a positive constant independent of $f$ and $x$.
\end{Lemma}
\begin{proof}
When $0<r\leq 1,$ $<r>\sim 1$, by usual Sobolev embedding
 \begin{align}
 H^3(\mathbb{B}_1)\hookrightarrow L^{\infty}(\mathbb{B}_1),
 \end{align}
 we can get \eqref{01}. When $r\geq 1$ , $<r>\sim |r|$,
 by the Sobolev embedding on $S^3$:
\begin{align}
H^{2}(S^3)\hookrightarrow L^{\infty}(S^3),
\end{align}
we get that
\begin{align}\label{123}
&r^3|f(x)|^2\nonumber\\
&\leq Cr^3\sum_{|\alpha|\leq 3}\int_{S^3}|\Omega^{\alpha}f(rw)|^2dw\nonumber\\
&\leq Cr^3\sum_{|\alpha|\leq 3}\int_{S^3}\int_{r}^{\infty}|\partial_{\rho}\Omega^{\alpha}f(\rho w)||\Omega^{\alpha}f(\rho w)|dwd\rho\nonumber\\
&\leq C\sum_{|\alpha|\leq 3}||\partial_r\Omega^{\alpha}f\Omega^{\alpha}f||_{L^1(\mathbb{R}^4\backslash\mathcal {K})}\nonumber\\
&\leq C\sum_{|a|\leq 3}||Z^{a}f||^2_{L^2(\mathbb{R}^4\backslash\mathcal {K})}.
\end{align}
Thus when $r\geq 1$, \eqref{01} still holds.
\end{proof}
\section{Lifespan Estimate of Classical Solutions to Problem \eqref{Quasilinear}}
In this section, we will prove Theorem \ref{Main12} by a bootstrap argument.\par
Let $u$ satisfy \eqref{Quasilinear}. For any $T>0$, denote
\begin{align}
M(T)&=\sup_{0\leq t\leq T}\sum_{|\mu|\leq 50}||\partial^{\mu}u(t)||_{L^2(\mathbb{R}^4\backslash \mathcal {K})}+(\log(2+T))^{-1/2}\sum_{|\mu|\leq 50}||<x>^{-1/2}\partial^{\mu}u||_{L^2_{t,x}(S_T)}\nonumber\\
&+\sum_{|\mu|\leq 50}||<x>^{-3/4}\partial^{\mu}u||_{L^2_{t,x}(S_T)}+\sup_{0\leq t\leq T}\sum_{|\mu|\leq 50}||\partial^{\mu}\partial u(t)||_{L^2(\mathbb{R}^4\backslash \mathcal {K})}\nonumber\\
&+\sup_{0\leq t\leq T}\sum_{|\mu|\leq 50}||\partial^{\mu}\partial^2 u(t)||_{L^2(\mathbb{R}^4\backslash \mathcal {K})}\nonumber\\
&+\sum_{|\mu|\leq 50}||<x>^{-3/4}\partial^{\mu}\partial u||_{L^2_{t,x}(S_T)}+\sum_{|\mu|\leq 50}||<x>^{-3/4}\partial^{\mu}\partial^2 u||_{L^2_{t,x}(S_T)}\nonumber\\
&+\sup_{0\leq t\leq T}\sum_{|\mu|\leq 49}||Z^{\mu}u(t)||_{L^2(\mathbb{R}^4\backslash \mathcal {K})}+(\log(2+T))^{-1/2}\sum_{|\mu|\leq 49}||<x>^{-1/2}Z^{\mu}u||_{L^2_{t,x}(S_T)}\nonumber\\
&+\sum_{|\mu|\leq 49}||<x>^{-3/4}Z^{\mu}u||_{L^2_{t,x}(S_T)}+\sup_{0\leq t\leq T}\sum_{|\mu|\leq 49}||Z^{\mu}\partial u(t)||_{L^2(\mathbb{R}^4\backslash \mathcal {K})}\nonumber\\
&+\sup_{0\leq t\leq T}\sum_{|\mu|\leq 49}||Z^{\mu}\partial^2 u(t)||_{L^2(\mathbb{R}^4\backslash \mathcal {K})}\nonumber\\
&+\sum_{|\mu|\leq 49}||<x>^{-3/4}Z^{\mu}\partial u||_{L^2_{t,x}(S_T)}+\sum_{|\mu|\leq 49}||<x>^{-3/4}Z^{\mu}\partial^2 u||_{L^2_{t,x}(S_T)}.
\end{align}
Assume
\begin{align}
M(0)\leq C_0\varepsilon.
\end{align}
We will prove that  if $\varepsilon>0$ is small enough, then for all $T\leq \exp{(\frac{c}{\varepsilon^2})}$ we have
\begin{align}
M(T)\leq 2A\varepsilon.
\end{align}
Here $A$ and $c$ are positive constants independent of $\varepsilon$, to be determined later.
\par
Assume that
\begin{align}
M(T)\leq 4A\varepsilon,
\end{align}
it follows from Lemma \ref{KS5} and Lemma \ref{KSS123} that
\begin{align}\label{main estimate}
M^2(T)
&\leq C_1^2\varepsilon^2+ C\sum_{|\mu|\leq 49}||<x>^{-\frac{1}{2}}
 Z^{\mu}F||^2_{L^2([0,T];L^{1,2}(|x|>\frac{3}{4}))}\nonumber\\
 &+C\sum_{\substack{|\mu|,|\nu|\leq 49\\|a|,|b|=0,1}}\int_0^{T}\int_{\mathbb{R}^4\backslash \mathcal {K}}(|Z^{\mu}\partial^{a}u'|+\frac{|Z^{\mu}\partial^{a}u|}{r})|Z^{\nu}\partial^{b}H|dxdt\nonumber\\
&+C\sum_{\substack{|\mu|,|\nu|\leq 49\\|a|,|b|=0,1}}\int_0^{T}\int_{\mathbb{R}^4\backslash \mathcal {K}}(|\partial \gamma|+\frac{|\gamma|}{r})|Z^{\mu}\partial^{a}u'|(|Z^{\nu}\partial^{b}u'|+\frac{|Z^{\nu}\partial^{b}u|}{r})dxdt\nonumber\\
&+C\sum_{\substack{|\mu|,|\nu|\leq 49\\|a|,|b|=0,1}}\sum_{\alpha,\beta=0}^{4} \int_0^{T}\int_{\mathbb{R}^4\backslash \mathcal {K}}(|Z^{\mu}\partial^{a}u'|+\frac{|Z^{\mu}\partial^{a}u|}{r})|[\gamma^{\alpha\beta}\partial_{\alpha\beta}, Z^{\nu}\partial^{b}]u|dxdt\nonumber\\
&+C\sum_{|\mu|\leq 50}||<x>^{-\frac{1}{2}}
 \partial^{\mu}F||^2_{L^2([0,T];L^{1,2}(|x|>\frac{3}{4}))}\nonumber\\
 &+C\sum_{\substack{|\mu|,|\nu|\leq 50\\|a|,|b|=0,1}}\int_0^{T}\int_{\mathbb{R}^4\backslash \mathcal {K}}(|\partial^{\mu}\partial^{a}u'|+\frac{|\partial^{\mu}\partial^{a}u|}{r})|\partial^{\nu}\partial^{b}H|dxdt\nonumber\\
&+C\sum_{\substack{|\mu|,|\nu|\leq 50\\|a|,|b|=0,1}}\int_0^{T}\int_{\mathbb{R}^4\backslash \mathcal {K}}(|\partial \gamma|+\frac{|\gamma|}{r})|\partial^{\mu}\partial^{a}u'|(|\partial^{\nu}\partial^{b}u'|+\frac{|\partial^{\nu}\partial^{b}u|}{r})dxdt\nonumber\\
&+C\sum_{\substack{|\mu|,|\nu|\leq 50\\|a|,|b|=0,1}}\sum_{\alpha,\beta=0}^{4} \int_0^{T}\int_{\mathbb{R}^4\backslash \mathcal {K}}(|\partial^{\mu}\partial^{a}u'|+\frac{|\partial^{\mu}\partial^{a}u|}{r})|[\gamma^{\alpha\beta}\partial_{\alpha\beta}, \partial^{\nu}\partial^{b}]u|dxdt\nonumber\\
&+C\sum_{\substack{|\mu|\leq 49\\|a|=0,1}}||\partial^{\mu}\partial^{a}F||^2_{L^{\infty}([0,T]; L^2(\mathbb{R}^4\backslash \mathcal {K}))}+C\sum_{\substack{|\mu|\leq 49\\|a|=0,1}}||\partial^{\mu}\partial^{a}F||^2_{L^2_{t,x}(S_{T})}\nonumber\\
&:=C_1^2\varepsilon^2+I+II+\cdots X.
\end{align}
Now we estimate all terms on the right-hand side of \eqref{main estimate}, respectively. For $I$, by H\"{o}lder inequality we have
\begin{align}
I &\leq C\sum_{|\mu|\leq 27}||<x>^{-\frac{1}{2}}Z^{\mu}u||^2_{L^2([0,T];L^{2,\infty}(|x|\geq \frac{3}{4}))}\sum_{|\mu|\leq 50}\sum_{|b|\leq 2}||Z^{\mu}\partial^{b}u||^2_{L^{\infty}([0,T];L^{2}(|x|\geq \frac{3}{4})}\nonumber\\
&\leq CA^2\varepsilon^2 \sum_{|\mu|\leq 27}||<x>^{-\frac{1}{2}}Z^{\mu}u||^2_{L^2([0,T];L^{2,\infty}(|x|\geq \frac{3}{4}))}.
\end{align}
It follows from the Sobolev embedding on $S^3$:
\begin{align}
H^{2}(S^3)\hookrightarrow L^{\infty}(S^3)
\end{align}
that
\begin{align}
&\sum_{|\mu|\leq 27}||<x>^{-\frac{1}{2}}Z^{\mu}u||^2_{L^2([0,T];L^{2,\infty}(|x|\geq \frac{3}{4}))}\nonumber\\
&\leq \sum_{|\mu|\leq 29}||<x>^{-\frac{1}{2}}Z^{\mu}u||^2_{L^2([0,T];L^{2}(|x|\geq \frac{3}{4}))}\nonumber\\
&\leq C(\log(2+T))A^2 \varepsilon^2.
\end{align}
Thus, we get
\begin{align}
I\leq C(\log(2+T))A^4 \varepsilon^4.
\end{align}
Similarly, we have
\begin{align}
V\leq C(\log(2+T))A^4 \varepsilon^4.
\end{align}
For $II$, noting that $0\in \mathcal {K}$ and then $1/r$ is bounded on $\mathbb{R}^4\backslash \mathcal {K}$, by H\"{o}lder inequality and Lemma \ref{spatial decay} we have
\begin{align}
II\leq & C\sum_{|\mu|\leq 49,|b|\leq 2}||<x>^{-\frac{3}{4}}Z^{\mu}\partial^{b}u||^2_{L^2_{t,x}(S_T)}\sum_{|\mu|\leq 27} ||<x>^{\frac{3}{2}}Z^{\mu}u||_{L^{\infty}([0,T];L^{\infty}(\mathbb{R}^4\backslash \mathcal {K}))}\nonumber\\
&\leq  C\sum_{|\mu|\leq 49,|b|\leq 2}||<x>^{-\frac{3}{4}}Z^{\mu}\partial^{b}u||^2_{L^2_{t,x}(S_T)}\sum_{|\mu|\leq 30}||Z^{\mu}u||_{L^{\infty}([0,T];L^{2}(\mathbb{R}^4\backslash \mathcal {K}))}\nonumber\\
&\leq CA^3\varepsilon^3.
\end{align}
Similarly,
\begin{align}
III, VI, VII\leq CA^3\varepsilon^3.
\end{align}
Noting the fact that $[\partial, Z]$ belongs to the span of $\{\partial\}$, for all $|\nu|\leq 49, |b|\leq 1, 0\leq \alpha,\beta\leq 4$, we have
\begin{align}
|[\gamma^{\alpha\beta}\partial_{\alpha\beta}, Z^{\nu}\partial^{b}]u|\leq \sum_{|\nu|\leq 49,|b|\leq 2}|Z^{\nu}\partial^bu|\sum_{|\nu|\leq 27}|Z^{\nu}u|.
\end{align}
By the same method used to deal with $II$, we can obtain
\begin{align}
IV\leq CA^3\varepsilon^3.
\end{align}
Similarly,
\begin{align}
VIII\leq CA^3\varepsilon^3.
\end{align}
For the last two terms,
it follows from H\"{o}lder inequality and Lemma \ref{spatial decay} that
\begin{align}
IX&\leq C\sum_{|\mu|\leq 50,|b|\leq 2}||\partial^{\mu}\partial^b u||^2_{L^{\infty}([0,T]; L^{2}(\mathbb{R}^4\backslash \mathcal {K}))}\sum_{|\mu|\leq 27}||\partial^{\mu}u||^2_{L^{\infty}([0,T]; L^{\infty}(\mathbb{R}^4\backslash \mathcal {K}))}\nonumber\\
&\leq CA^2\varepsilon^2 \sum_{|\mu|\leq 30}||Z^{\mu}u||^2_{L^{\infty}([0,T]; L^{2}(\mathbb{R}^4\backslash \mathcal {K}))}\nonumber\\
&\leq CA^4\varepsilon^4
\end{align}
and
\begin{align}
X&\leq  C\sum_{|\mu|\leq 50,|b|\leq 2}||<r>^{-\frac{3}{4}}\partial^{\mu}\partial^b u||^2_{L^{2}([0,T]; L^{2}(\mathbb{R}^4\backslash \mathcal {K}))}\sum_{|\mu|\leq 27}||<r>^{\frac{3}{4}}\partial^{\mu}u||^2_{L^{\infty}([0,T]; L^{\infty}(\mathbb{R}^4\backslash \mathcal {K}))}\nonumber\\
&\leq  C\sum_{|\mu|\leq 50,|b|\leq 2}||<r>^{-\frac{3}{4}}\partial^{\mu}\partial^b u||^2_{L^{2}([0,T]; L^{2}(\mathbb{R}^4\backslash \mathcal {K}))}\sum_{|\mu|\leq 30}||Z^{\mu}u||^2_{L^{\infty}([0,T]; L^{2}(\mathbb{R}^4\backslash \mathcal {K}))}\nonumber\\
&\leq CA^4\varepsilon^4.
\end{align}
 Combining all the previous estimates, we get
\begin{align}
M(T)\leq C_{1}\varepsilon+ C(\log(2+T))^{1/2}A^2\varepsilon^2+CA^{3/2}\varepsilon^{3/2}+CA^2\varepsilon^2.
\end{align}
Taking $A=2\max\{C_0,C_1\}$, we see that if
\begin{align}
C(\log(2+T))^{1/2}A\varepsilon,~ CA^{1/2}\varepsilon^{1/2},~ CA\varepsilon\leq \frac{1}{2},
\end{align}
then
\begin{align}
 M(T)\leq 2A\varepsilon.
\end{align}
From the above argument, we know that if the parameter $\varepsilon>0$ is small enough, then for all $T\leq \exp(\frac{c}{\varepsilon^2})$, we have
\begin{align}
M(T)\leq 2A\varepsilon.
\end{align}
Consequently, we get the lifespan estimate of classical solutions to problem \eqref{Quasilinear}:
\begin{align}
T_{\varepsilon}\geq \exp(\frac{c}{\varepsilon^2}),
\end{align}
where $c$ is a positive constant independent of $\varepsilon$.
\section*{Acknowledgements}
The authors would like to express their sincere gratitude to Professor Ta-Tsien Li for his helpful suggestions and encouragement.

\end{document}